\documentclass[12pt,leqno,a4paper]{amsart}
\usepackage{amsmath,amsthm,amsfonts,amssymb,latexsym, mathabx}
\usepackage{enumerate}

\usepackage{color, colortbl, xcolor, makecell}

\newcommand{\FF}{{\mathbb{F}}}

\newcommand{\bG}{{\mathbf{G}}}
\newcommand{\OO}{{\operatorname{O}}}

\newcommand{\SL}{{\operatorname{SL}}}

\newcommand{\PSL}{{\operatorname{PSL}}}

\newcommand{\POmega}{{\operatorname{P\Omega}}}

\newcommand{\Syl}{{\operatorname{Syl}}}

\newcommand{\alt}{\mathfrak{A}}
\newcommand{\sym}{\mathfrak{S}}

\newcommand{\tw}[1]{{}^#1\!}

\def\syl#1#2{{\rm Syl}_#1(#2)}

\def\oh#1#2{{\bf O}_{#1}(#2)}

\def\zent#1{{\bf Z}(#1)}
\def\irr#1{{\rm Irr}(#1)}
\def\norm#1#2{{\bf N}_{#1}(#2)}
\def\cent#1#2{{\bf C}_{#1}(#2)}

\def\irrp#1#2{{\rm Irr}_{#1'}(#2)}

\def\sbs{\subseteq}

\newtheorem{thm}{Theorem}[section]
\newtheorem{lem}[thm]{Lemma}

\newtheorem{conj}[thm]{Conjecture}
\newtheorem{prop}[thm]{Proposition}

\newtheorem*{thm*}{Theorem $\heartsuit$}

\newtheorem*{conjA}{Conjecture A}
\newtheorem*{conjB}{Conjecture B}

\newtheorem*{conjC}{Conjecture C}
\newtheorem*{thmD}{Theorem D}
\newtheorem*{thmE}{Theorem E}

\theoremstyle{remark}

\newcommand\wt[1]{\widetilde{#1}}

\usepackage{array}
\newcolumntype{?}{!{\vrule width 1pt}}

\newcommand{\GL}{\operatorname{GL}}

\newcommand{\PSp}{\operatorname{PSp}}
\newcommand{\Sp}{\operatorname{Sp}}
\newcommand{\SO}{\operatorname{SO}}
\newcommand{\GO}{\operatorname{GO}}

\newcommand{\Aut}{\operatorname{Aut}}
\newcommand{\aut}{\operatorname{Aut}}
\newcommand\type[1]{\operatorname{#1}}

\begin{document}

\title[Principal Blocks for different Primes, II]{Principal Blocks for different Primes, II}

\author{Gabriel Navarro}
\address[G. Navarro]{Departament de Matem\`atiques, Universitat de Val\`encia, 46100 Burjassot,
       Val\`encia, Spain.}
\email{gabriel@uv.es}

\author{Noelia Rizo}
\address[N. Rizo]{Departamento de Matem\'aticas, Universidad de Oviedo, 33007, Oviedo, Spain}
\email{rizonoelia@uniovi.es}

\author{A. A. Schaeffer Fry}
\address[A. A. Schaeffer Fry]{Department of  Mathematics and Statistics, Metropolitan State University of Denver, Denver, CO 80217, USA}
\email{aschaef6@msudenver.edu}

\thanks{The first and second authors are partially supported by Grant PID2019-103854GB-I00 funded by MCIN/AEI/10.13039/501100011033. The second author also acknowledges support by Generalitat Valenciana AICO/2020/298 and Grant PID2020-118193GA-I00 funded by MCIN/AEI/10.13039/501100011033.  The third author is partially supported by a grant from the National Science Foundation, Award No. DMS-2100912.}



\begin{abstract}
If $G$   is a finite group, we have proposed new conjectures on the interaction
between different primes and their corresponding Brauer principal blocks.
In this paper, we give strong support to the validity of these conjectures.   \end{abstract}

\medskip

\maketitle

\centerline{\sl To Pham Huu Tiep, on his 60th birthday}
\bigskip

\section{Introduction}
Meaningful interaction between the representation theory of finite groups from the perspective of different primes
is extremely rare. However, in \cite{NRS1}, we proposed the following three plausible conjectures, which 
extended work of several authors (see \cite{LWXZ}, \cite{NW}, \cite{BNOT}, \cite{MN20}).

\medskip
If $p$ is  a prime and   $G$ is a finite group, we denote by $B_p(G)$  the principal $p$-block of $G$.
The main subject of our work is the set $\irrp{p}{B_p(G)}$ of the irreducible
complex characters in the principal $p$-block of $G$ whose degree is not divisible by $p$.
This set seems to possess remarkable properties.
\medskip
 
\begin{conjA}
Let $G$ be a finite group and let $p$ and $q$ be different primes.
If $$\irrp{p}{B_p(G)} \cap  \irrp{q}{B_q(G)}=\{1_G\},$$
then there are a Sylow $p$-subgroup $P$ of $G$
and a Sylow $q$-subgroup $Q$ of $G$ such that $xy=yx$ for all
$x \in P$ and $y \in Q$.
\end{conjA}
\medskip
\begin{conjB}
Let $G$ be a finite group and let $p$ and $q$ be primes dividing the order of $G$.
If $\irrp{p}{B_p(G)}=\irrp{q}{B_q(G)}$, then $p=q$.
\end{conjB}
\medskip
\begin{conjC}
Let $G$ be a finite group, and let $p$ and $q$ be different primes.
Then $q$ does not divide $\chi(1)$ for all $\chi \in \irrp{p}{B_p(G)}$
and $p$ does not divide $\chi(1)$ for all $\chi \in \irrp{q}{B_q(G)}$
if and only if there are a Sylow $p$-subgroup $P$ of $G$
and a Sylow $q$-subgroup $Q$ of $G$ such that $xy=yx$ for all
$x \in P$ and $y \in Q$.
\end{conjC}
\medskip
The main result of \cite{NRS1} was to reduce Conjecture A to a problem on almost simple groups and to prove it in the case that one of the primes is $2$.  To prove Conjecture A for almost simple groups in the case 
where $p$ and $q$ are both odd
remains quite a challenge.  

\medskip

In the present paper, we focus on Conjectures B and C.
Using the Classification of Finite Simple Groups, in
our first main theorem we prove the following.

\begin{thmD}
Conjecture C implies Conjecture B.
\end{thmD}

After proving Theorem D, therefore, we concentrate our efforts in the remainder of the paper towards  Conjecture C.

\begin{thmE}
Conjecture C holds for finite simple groups. 
\end{thmE}

Besides Conjecture C being true for simple groups,
in Theorem \ref{psolv} below, we shall also prove that Conjecture C is true for
$p$-solvable groups, assuming the inductive Alperin--McKay condition.
This gives strong support to the validity of
 this conjecture. 
 \medskip
 
 Unfortunately, at the time of this writing, we still do not know how to reduce Conjecture C to a question on almost simple groups.
\section{Theorem D}\label{sectionreduction}
   
 In this Section we prove that Conjecture C implies Conjecture B. % assuming the following result on simple groups, which we shall prove in the next section.
This will require the following result on simple groups.

 \begin{thm}\label{thsimples} Let $p,q$ be different primes and let $S$ be a non-abelian simple group with $pq\mid |S|$. Assume that $[P,Q]=1$ for some
 Sylow $p$-subgroup $P$ of $S$ and a Sylow $q$-subgroup $Q$ of $S$.
 Then one of the following holds:
 \begin{enumerate}[(a)]
 \item There exists $\alpha\in{\rm Irr}_{p'}(B_p(S))-  {\rm Irr}(B_q(S))$  which
 is ${\rm Aut}(S)_p$-invariant, where ${\rm Aut}(S)_p$ is some Sylow $p$-subgroup of ${\rm Aut}(S)$.

 \item There exists $\alpha\in{\rm Irr}_{q'}(B_q(S))- {\rm Irr}(B_p(S))$ which
 is  ${\rm Aut}(S)_q$-invariant, where ${\rm Aut}(S)_q$ is some Sylow $q$-subgroup of ${\rm Aut}(S)$.
 \end{enumerate}
  \end{thm}
  \begin{proof}
  First, note that by \cite[Lemma 3.1]{MN20}, the condition $[P,Q]=1$ implies that $[\bar P, \bar Q]=1$ for some Sylow $p$- and $q$- subgroups $\bar P, \bar Q$ of any covering group of $S$. If $S$ is one of the sporadic groups $J_1$ or $J_4$, then we may use GAP and its Character Table Library to see that the statement holds. 
  Then \cite[Propositions 3.2-3.4]{MN20} further imply that we may assume $S$ to be a simple group of Lie type that is not isomorphic to a sporadic or alternating group and is defined in characteristic $r_0\not\in\{p,q\}$.
  
  So, let $S$ be of the form $S=G/\zent{G}$ for $G$ a group of Lie type of simply connected type defined in  characteristic $r_0\neq p,q$.  %We first suppose that $r\not\in\{p, q\}$. 
  In this case, we may further assume that $p$ and $q$ are both odd and that the Sylow $p$- and $q$-subgroups of $G$  are abelian, using \cite[Proposition 3.5]{MN20}. %(////from same prop, also have $P,Q$ abelian and $d_p(r^a)=d_q(r^q)$ where $G$ def over $\FF_{r^a}$, but I don't think I'll need this///\textcolor{blue}{yes I will! to get rid of (ii) below})  
 
 Since $p$ and $q$ are odd, note that we may therefore assume without loss that $p\geq 5$.  Further, using \cite[Lemma 2.1 and Proposition 2.2]{malle14}, we have $p\nmid |\zent{G}|$. Together, this implies that $p$ does not divide the order of any diagonal or graph outer automorphism, so that
 $\Aut(S)_p$ may be taken as a subgroup of $S\rtimes \langle F\rangle$, where $F$ is a generating field automorphism. %; or
% $S=\PSL_n^\epsilon(r^a)$, $p\mid |\wt{S}/S|=\gcd(n, r^a-\epsilon)$ where $\wt{S}:=\PGL_n^\epsilon(r^a)$, and $\Aut(S)_p$ may be taken as a subgroup of $\wt{S}\rtimes \langle F\rangle$. 
%However, the condition that the Sylow $p$-subgroups of $G$ are abelian forces that $p>n$ in the latter situation, using \cite[Proposition 2.2]{malle14}, a contradiction.  Hence, $\Aut(S)_pS/S$ can be considered as a subgroup of $\langle F\rangle$.

Now, in \cite[Proposition 3.5]{NRS1}, building off of \cite[Theorem 3.5]{GSV19}, we have shown that there is a character $\chi\in\irrp{p}{B_p(S)}$ such that $\chi$ is the deflation of a so-called semisimple character $\chi_t$ of $G$ with $t\neq 1$ an element of the dual group $G^\ast$ of order a power of $p$. In particular, $\chi_t$ lies in the Lusztig series $\mathcal{E}(G, t)$. Since $B_q(G)$ contains only characters lying in Lusztig series $\mathcal{E}(G, s)$ with $|s|$ a power of $q$ (see \cite[Theorem 9.12]{CE04}), we see that $\chi$ does not  lie in $\irr{B_q(S)}$.

It now suffices to argue that $t$ can have been chosen so that $\chi_t$ is invariant under $\aut(S)_p$. Following the proof of \cite[Theorem 3.5]{GSV19}, we see that $t$ was chosen to lie in $\zent{P^\ast}$, where $P^\ast$ is a Sylow $p$-subgroup of the dual group $G^\ast$, and that we may further choose $t$ to have order $p$ without losing the properties we have already discussed.  

Now, since $p\nmid |\zent{G}|$, we have $\cent{\bG^\ast}{t}$ is connected, where $\bG^\ast$ is the ambient reductive group whose fixed points under an appropriate Frobenius endomorphism yields $G^\ast$ (see \cite[Exercise 20.16]{MT11}). Then for $\varphi\in\langle F\rangle$, we have $\chi_t^\varphi=\chi_{t^{\varphi^\ast}}$, where $\varphi^\ast$ is an appropriate field automorphism of $G^\ast$ (see \cite[Corollary 2.5]{NTT08}). Since $t^{\varphi^\ast}$ must be $G^\ast$-conjugate to another element of order $p$ in $\zent{P^\ast}$, we have $\chi_t^\varphi$ is also of the form $\chi_{t'}$ with $t'\in\zent{P^\ast}$ of order $p$, is still trivial on $\zent{G}$, and still lies in $B_p(S)$ but not $B_q(S)$ under deflation.  Hence $\aut(S)_p$ acts on such characters and we let $\alpha$ be a fixed point under this action, since  the number of such elements $t'$ is relatively prime to $p$. 
  \end{proof}
 
 \medskip

 \begin{lem}\label{perm}
 Suppose that $N$ is a minimal normal subgroup of $G$, which is a direct product of the different
 $G$-conjugates of a non-abelian simple group $S$. Let $P \in \syl pG$.
Suppose that  $\alpha\in{\rm Irr} (S)$   
 is ${\rm Aut}(S)_p$-invariant, where ${\rm Aut}(S)_p$ is some Sylow $p$-subgroup of ${\rm Aut}(S)$.
 Then there are $g_i \in G, h_i \in P$ and $\sigma_i \in {\rm Aut}(S^{g_i})$ such that 
 $(\alpha^{g_1})^{\sigma_1h_1} \times \cdots \times  (\alpha^{g_m})^{\sigma_m h_m} \in \irr N$ is $P$-invariant.
 \end{lem}
  
 \begin{proof}
 Suppose that $N$ is the direct product of $\Omega=\{ S^g |\, g \in G\}$.
 Now, write
 $$\Omega=\mathcal O(S_1) \cup \ldots \cup \mathcal O(S_t), $$
 where $\mathcal O(S_i)$ is the $P$-orbit of some $S_i=S^{x_i}$, for some $x_i \in G$.
 Then $N=N_1 \times \cdots \times N_t$, where $N_i$ is the product of the elements in $\mathcal{O}(S_i)$.
 Of course, $N_i$ is $P$-invariant.
 Let $\alpha_i=\alpha^{x_i} \in \irr{S_i}$. Suppose that $\{S_i^{y_1}, \ldots S_i^{y_r}\}$ are 
 the different $P$-conjugates of $S_i$, where $y_i \in P$.  Hence
 $N_i=S_i^{y_1}\times  \ldots \times S_i^{y_r}$, and
 $$P=\bigcup_{j=1}^r \norm P{S_i} y_j$$
 is a disjoint union.
 
 By hypothesis, $\alpha$ is invariant under $X \in \syl p{{\rm Aut}(S)}$.
 Therefore $\alpha_i$ is invariant under $X_i=X^{x_i} \in \syl p{{\rm Aut}(S_i)}$.
 Let $M_i=\norm G{S_i}$ and $C_i=\cent G{S_i}$. We have that $M_i/C_i$ embeds into ${\rm Aut}(S_i)$.
 Then $\norm P{S_i}C_i/C_i$ is a $p$-subgroup of ${\rm Aut}(S_i)$, and therefore, there
 is $\sigma_i \in {\rm Aut}(S_i)$ such that $(\norm P{S_i}C_i/C_i)\sbs X_i^{\sigma_i} $.
 Since $\alpha_i$ is $X_i$-invariant, it follows that $\beta_i=(\alpha_i)^{\sigma_i}\in{\rm Irr}(S_i)$ is $X_i^{\sigma_i}$-invariant, and therefore
 $\norm P{S_i}$-invariant.  We claim that $\gamma_i=\beta_i^{y_1} \times \cdots \times \beta_i^{y_r}\in\irr{N_i}$ is $P$-invariant.
 Indeed, if $x \in P$, then $y_k x=w_k y_{\sigma(k)}$ for some $w_k \in \norm P{S_i}$, $1 \le k \le r$,
  and $\sigma$ a permutation of $\sf S_r$.
  Now, if $u \in S_i$, then 
  $$\gamma_i^{x^{-1}}(u^{y_k})=\gamma_i(u^{y_k x})=\alpha(1)^{r-1}\beta_i^{y_{\sigma(k)}}(u^{w_ky_{\sigma(k)}})=
  \alpha(1)^{r-1}\beta_i(u^{w_k})= \alpha(1)^{r-1}\beta_i(u)=\gamma_i(u^{y_k}) \, .$$
  This proves that $\gamma_i$ is $P$-invariant. Hence $\gamma=\gamma_1 \times \cdots \times \gamma_t \in \irr N$
  is $P$-invariant.
  \end{proof}
 
% Checked for the sporadics and all simple groups in AllCharacterTableNames with GAP.

 We will need the following well-known result of   J. Alperin  and E. C. Dade. 
\begin{thm}\label{isomblocks}
Suppose that $N$ is a normal subgroup of $G$, with $G/N$ a $p'$-group.
Let $P \in \syl p G$ and assume that $G=N\cent GP$. Then restriction of characters defines
a natural bijection between the irreducible characters of the principals blocks
of $G$ and $N$. %In particular, $|{\rm Irr}_{p',\tau}(B_0(G))|=|{\rm Irr}_{p', \tau}(B_0(N))|$, for any  $\tau \in {\rm Gal}(\Q^{\rm ab}/\Q)$.
\end{thm}

\begin{proof}
The case where $G/N$ is solvable was proved in \cite{A} and the general case
in \cite{D}. %The last part of the statement follows immediately since $\tau$ acts on $\irr{B_0}$ (preserving character degrees). 
\end{proof}

\begin{thm}
Assume that Conjecture C is true for all finite groups.
Let $p$ and $q$ be primes. Assume that $G$ is a finite group of order divisible by $p$ and $q$. If ${\rm Irr}_{p'}(B_p(G))={\rm Irr}_{q'}(B_q(G))$,   then $p=q$.
\end{thm}

\begin{proof}
We argue by induction on $|G|$. Assume that $p\ne q$.
 By Conjecture C, we know that $[P,Q]=1$ for some $P \in \syl pG$ and $Q \in \syl q G$.

\smallskip

\textit{Step 0. If $1\neq N\lhd G$, then $p$ divides $|N|$ or $q$ divides $|N|$.} 

Otherwise by \cite[Theorem 9.9(c)]{nbook} $${\rm Irr}_{p'}(B_p(G/N))={\rm Irr}_{p'}(B_p(G))={\rm Irr}_{q'}(B_q(G))={\rm Irr}_{q'}(B_q(G/N))$$ and by induction we are done. 

\smallskip

\textit{Step 1. Let $L$ be a proper normal subgroup of $G$.
Then $G/L$ has order  divisible by $p$ or $q$.} 
\smallskip

Suppose that $G/L$ has $p'$ and $q'$-order.
We claim that ${\rm Irr}_{p'}(B_p(L))={\rm Irr}_{q'}(B_q(L))$. Indeed, let $\theta\in{\rm Irr}_{q'}(B_q(L))$.
Then there exists $\chi\in{\rm Irr}(B_q(G))$ over $\theta$. Then $\chi$
has $q'$-degree by  \cite[Theorem 11.29]{isbook} and
therefore  $\chi\in{\rm Irr}_{p'}(B_p(G))$, by hypothesis.
 Therefore  $\theta\in{\rm Irr}_{p'}(B_p(L))$. 
 By symmetry, the claim is proved. Therefore,
 the proof of Step 1 follows by using the inductive hypothesis.
 
 \smallskip

\textit{Step 2. Let $N$ be a minimal normal subgroup of $G$ and suppose that $N$ is 
an elementary abelian $p$-group.  Then $G=\cent GN\cent GQ$.} 
\smallskip

Write $M=\cent G N$ and $L=M\cent G Q$. We have that
$Q\subseteq\cent G P\subseteq\cent G N=M$ and $G/M$ is a $q'$-group. Also by the Frattini argument $G=M\norm G Q$ and $L\lhd G$. Notice that $G/L$ is a $q'$-group and a $p'$-group (because $P\subseteq \cent G Q\subseteq L$). Then we use Step 1.

\smallskip

\textit{Step 3. Let $N$ be a minimal normal subgroup of $G$ and suppose that $N$
is abelian.  Then $G=\cent GN$.}

\smallskip

By Step 0, we may assume that $N$ is a $p$-group or
a $q$-group. By symmetry, assume that $N$ is a $p$-group. Write $M=\cent G N$.
We prove first that $B_p(G)$ is the only $p$-block of $G$ covering $B_p(M)$. Let $B$ be a $p$-block of $G$ covering $B_p(M)$ and let $D$ be a defect group of $B$. Then $N\subseteq D$ by  \cite[Theorem 4.8]{nbook} and $\cent G D\subseteq M=\cent G N$. Then $B$ is regular with respect to $M$ (\cite[Lemma 9.20]{nbook}) and hence by \cite[Theorem 9.19]{nbook} we have that $B=B_p(M)^G=B_p(G)$ by the third main theorem (see \cite[Theorem 6.7]{nbook}). Hence $B_p(G)$ is the only $p$-block of $G$ covering $B_p(M)$.
In particular, we have that  ${\rm Irr}(G/M)\subseteq{\rm Irr}(B_p(G))$.

 Next, we prove that $G=M$. Recall that $G/M$ is a $q'$-group, because $[Q,N]=1$.
 By Step 2,  we have $G=M\cent G Q$. Hence by Theorem \ref{isomblocks}, we have 
 that the restriction map

$${\rm res}:{\rm Irr}(B_q(G))\rightarrow{\rm Irr}(B_q(M))$$ is a bijection. 
We claim that ${\rm Irr}_{p'}(G/M)=\{1_G\}$. Indeed, if $\chi\in{\rm Irr}_{p'}(G/M)$, then $\chi\in{\rm Irr}_{p'}(B_p(G))={\rm Irr}_{q'}(B_q(G))$. By Theorem \ref{isomblocks}, $\chi_M$ is irreducible and since $\chi$ lies over $1_M$ we have $\chi_M=1_M$. Hence $\chi=1_G$ by the uniqueness
of the restriction map. Thus ${\rm Irr}_{p'}(G/M)=\{1_{G/M}\}$ and $G=M$ by \cite[Lemma 2.2]{NRS1}.

\smallskip

\textit{Step 4. Let $N$ be a minimal normal subgroup of $G$, then $N$ is not  abelian.} 
\smallskip

Suppose the contrary and assume without loss of generality that $N$ is an elementary abelian $p$-group, so by Step 3 we have $G=\cent G N$ and $N\subseteq\zent G$. By \cite[Theorem 9.10]{nbook} we have that $B_p(G/N)$ is the unique $p$-block of $G/N$ contained in $B$. We claim that ${\rm Irr}_{p'}(B_p(G/N))={\rm Irr}_{q'}(B_q(G/N))$. Indeed, we have that $${\rm Irr}_{p'}(B_p(G/N))\subseteq{\rm Irr}_{p'}(B_p(G))={\rm Irr}_{q'}(B_q(G))={\rm Irr}_{q'}(B_q(G/N))$$ where we have used \cite[Theorem 9.9]{nbook} in the last equality. On the other hand, let $\chi\in{\rm Irr}_{q'}(B_q(G/N))$, so $\chi\in{\rm Irr}_{q'}(B_q(G))={\rm Irr}_{p'}(B_p(G))$ and $N\subseteq{\rm Ker}(\chi)$. Since $B_p(G/N)$ is the only $p$-block of $G/N$ contained in $B_p(G)$, we have that $\chi\in{\rm Irr}_{p'}(B_p(G/N))$ and the claim is proven. By using the inductive hypothesis,
we have that $p$ does not divide $|G/N|$.  Therefore,
$\{1_G\}={\rm Irr}(B_p(G/N))={\rm Irr}_{p'}(B_p(G/N))={\rm Irr}_{q'}(B_q(G/N))$ and $q\nmid |G/N|$ by  \cite[Lemma 2.1]{NRS1}). Hence $q\nmid |G|$ and this is a contradiction.

\smallskip 
\textit{Step 5. Let $N$ be a minimal normal subgroup of $G$, then $pq$ divides $|N|$.}

 \smallskip
 Suppose that $N$ is a $p'$-group. We claim first that $NQ$ does not have a normal $q$-complement. Indeed, suppose the contrary and let $X\lhd NQ$ be a normal $q$-complement. Then $N\cap X$ is a normal $q$-complement of $N$ and by the minimality of $N$ we have that either $N\cap X=1$ or $N\cap X=N$. If $N\cap X=N$, $N$ is $q'$ and $p'$, contradiction with Step 0. If $N\cap X=1$ then $N\cong XN/X$ is a $q$-group, which is a contradiction with Step 4. Therefore $NQ$ does not
 have a normal $q$-complement. By \cite[Corollary 3]{IS} there is $\tau \in \irr{B_q(QN)}$
 non-linear of $q'$-degree. Therefore $1\ne \tau_N\in{\rm Irr}_{q'}(B_q(N))$. By \cite[Lemma 4.3]{murai} we have that there is  some $\gamma \in {\rm Irr}_{q'}(B_q(G))$ lying over $\tau_N$.
 By hypothesis, we have that $\gamma$ is in the principal $p$-block
 of $G$, and therefore $\tau_N$ is in the principal $p$-block of $N$, which is a contradiction since $N$ is a $p'$-group and $\tau_N\neq 1$.

 \smallskip
 
 \textit{Final Step.} If $N$ is a minimal normal subgroup of $G$, then $N$ is semisimple by Step 4. 
 Suppose that $N$ is a direct product of all the different $G$-conjugates of a simple group
 $S$ of order divisible by $pq$.  Suppose that (a) of Theorem \ref{thsimples} holds and let $\alpha$ be the character in ${\rm Irr}_{p'}(B_p(S))$ (not in ${\rm Irr}(B_q(S))$) which is ${\rm Aut}(S)_p$-invariant, given by Theorem \ref{thsimples}. 
 Notice that any $G$-conjugate or ${\rm Aut}(G)$-conjugate of $\alpha$ is in the principal $p$-block and not in the principal $q$-block os $S$.
By Lemma \ref{perm}, there exists
$\tau \in{\rm Irr}_{p'}(B_p(N))$ which is $P$-invariant, and such that each of its
factors does not belong to the principal $q$-block. In particular,
$\tau$ does not belong to the principal $q$-block of $N$. Now $\tau$ extends to $PN$ by \cite[Corollary 8.16]{isbook2}.
By  \cite[Lemma 4.3]{murai} there is $\chi\in{\rm Irr}_{p'}(B_p(G))$ lying over $\tau$. Then $\chi\in{\rm Irr}_{q'}(B_q(G))$ and thus $\tau\in{\rm Irr}_{q'}(B_q(N))$, which is a contradiction. Assuming (b) in Theorem \ref{thsimples} and reasoning analogously we get again a contradiction. \end{proof}

\section{Conjecture C and $p$-solvable groups}

As of the writing of this article, we can only prove Conjecture C for $p$-solvable 
 groups if the following (expected) consequence of the Alperin--McKay conjecture is true.

\begin{conj}\label{overMcKay}
Suppose that $N$ is normal in $G$, and let $P \in \syl pN$.
Then there is a bijection $$^*:{\rm Irr}_{p'}(B_p(N)) \rightarrow {\rm Irr}_{p'}(B_p(\norm NP))$$
such that for each $\theta \in {\rm Irr}_{p'}(B_p(N))$, there is a bijection
$f_\theta: \irr{B_p(G)|\theta} \rightarrow \irr{B_p(\norm GP)|\theta^*}$ such that
$\chi(1)/\theta(1)=f_\theta(\chi)/\theta^*(1)$ for all $\chi \in \irr{G|\theta}$.
\end{conj}

(In fact, it is expected that $^*$ is $\norm GP$-equivariant, and that
the character triples $(G_\theta, N, \theta)$ and $(\norm GP_\theta, \norm NP, \theta^*)$
are isomorphic.)

In any case, we shall only need Conjecture \ref{overMcKay} in the case where
$G/N$ is a $p'$-group.

In our proof, we shall also need a McKay divisibility theorem, which was made possible
after M. Geck proved a remarkable conjecture on Glauberman correspondents. %We need the following easy observation.

%\begin{lem}\label{extension}Let $G$ be a finite group and let $L$ be a normal $p$-subgroup of $G$. Suppose that $\mu\in \irr L$ is $G$-invariant. If there is  a character $\chi \in {\rm Irr}_{p'}(G\, |\, \mu)$, then $\mu$ extends to $G$.
%\end{lem}
%\begin{proof}
%Observe that $\mu$ is linear. Let $P$ be a Sylow $p$-subgroup of $G$. Since $\chi$ is of $p'$-degree, we have that $\chi_P$ has a linear irreducible constituent, which lies over $\mu$ since $\mu$ is $G$-invariant. Hence $\mu$ extends to $P$. Now the result follows by \cite[Theorem 6.26]{isbook2}.
%\end{proof}

\medskip

In the following, we follow the proof of \cite[Theorem A]{Ri19}, and then use Geck's result.

\begin{thm}\label{mckaydiv}
Let $G$ be a $p$-solvable group and let $P\in{\rm Syl}_p(G)$. Then there is a bijection $$^*:{\rm Irr}_{p'}(G) \rightarrow {\rm Irr}_{p'}(\norm GP)$$
such that $\chi^*(1)$ divides $\chi(1)$ and $\chi(1)/\chi^*(1)$ divides $|G:\norm GP|$.
\end{thm}

\begin{proof} We argue by induction on $|G|$. As in the proof of \cite[Theorem A]{Ri19} we may assume that  $\oh {p} G =1$, and hence $K=\oh {p'} G >1$.

Let $S/K=\oh {p}{G/K}$, and notice that $P_0=P\cap S$ is a Sylow $p$-subgroup of $S$. 
By the Frattini argument we have that $G=K\norm G {P_0}$ and  $\norm G {P_0}<G$ since $\oh{p}G=1$. Let $\theta_1, \ldots, \theta_s$ be a complete set of representatives of the orbits of the action of $\norm G P$ on the $P$-invariant irreducible characters of $K$. By \cite[Lemma 9.3]{Na18}, we have that
$${\rm Irr}_{p'} (G) = {\rm Irr}_{p'} (G\,|\,\theta_1) \cup \cdots \cup {\rm Irr}_{p'} (G\,|\,\theta_s)$$
is a disjoint union.  Fix $\theta_i \in \irr K$ and observe  that  $\theta_i$ is also $P_0$-invariant. Let $\theta_i^*\in \irr{\cent K {P_0}}$ be the Glauberman correspondent of $\theta_i$ (see \cite[Theorem 2.9]{Na18}, for instance) and let $T_i$ be the stabilizer of $\theta_i$ in $G$. Since the Glauberman correspondence and the action of $\norm G {P_0}$ commute (see \cite[Lemma 2.10]{Na18}), it follows that $ \norm {T_i} {P_0} =T_i\cap \norm G {P_0}$ is the stabilizer of $\theta_i^*$ in $\norm G {P_0}$.

Again as in the proof of \cite[Theorem A]{Ri19} we obtain a bijection 

%By  Dade's theorem we have that $(T_i, K, \theta_i)$ and $ (\norm {T_i} {P_0}, \cent K {P_0}, \theta_i^*)$ are isomorphic character triples. Hence there is a bijection $ ^*: \irr{T_i\, |\, \theta_i} \rightarrow \irr{ \norm {T_i} {P_0} \, |\, \theta_i^*}$ such that $ \psi(1)/\theta_i(1)= \psi^*(1)/\theta_i^*(1)$ for all  $\psi \in  \irr{T_i\, |\, \theta_i} $, and since $\theta_i(1)$ and $\theta_i^*(1)$ are $p'$-numbers we have that 

$$ ^*: {\rm Irr}_{p'}(T_i\, |\, \theta_i) \rightarrow {\rm Irr}_{p'} ( \norm {T_i} {P_0} \, |\, \theta_i^*)$$ satisfying $\psi(1)/\psi^*(1)=\theta_i(1)/\theta_i^*(1)$. Hence we have that $\psi^*(1)$ divides $\psi(1)$ by the main result of \cite{Geck}. Moreover, since $G=\norm G{P_0}K$ we have that $|K:\cent K{P_0}|=|T_i:\norm {T_i}{P_0}|$ and since $\theta_i(1)/\theta_i^*(1)$ divides $|K:\cent K{P_0}|$ we conclude that $\psi(1)/\psi^*(1)$ divides $|T_i:\norm {T_i}{P_0}|$.

%Proof of Dade's in (see \cite[Theorem 6.5]{T1})

\smallskip

The remaining part of the proof proceeds exactly as in the proof of \cite[Theorem A]{Ri19}. 
\end{proof}

\begin{thm}\label{psolv}
Let $G$ be a finite $p$-solvable group.
If Conjecture \ref{overMcKay} is true, then
Conjecture C is true for $G$.
\end{thm}

\begin{proof}
By \cite{MN20}, we only need to prove that if all ${\rm Irr}_{p'}(B_p(G))$ have $q'$-degree
and all ${\rm Irr}_{q'}(B_q(G))$ have $p'$-degree, then $[P,Q]=1$ for some
$P \in \syl pG$ and $Q \in \syl qG$. Let $N$ be a normal subgroup of $G$. Since the hypothesis is satisfied by $G/N$, by induction, we know that $[P,Q] \sbs N$ for some $P \in \syl pG$ and $Q\in \syl qG$.

%\textcolor{blue}{We may assume that $G$ has a unique minimal normal subgroup $N$.} 

Suppose that $\oh{p'}G=1$.  Then we know that
$\irr{B_p(G)}=\irr G$, by Theorem 10.20 of \cite{nbook}. Hence, all the irreducible characters in $\irrp{p}G$ have $q'$-degree. Let $L=\oh pG$ and let $P \in \syl pG$ and $Q\in \syl qG$ such that $[P,Q]\sbs L$.  Therefore $Q$ normalizes $P$ and $|G:\norm GP|$ is not divisible by $q$. By Theorem \ref{mckaydiv}, we have that
all $\irr{\norm GP/P'}$ have $q'$-degree. By the It\^o-Michler theorem we have that $QP'$ is normal in $\norm G P$, and hence $[QP',P]\subseteq P'$. Then $Q$ acts trivially on $P/P'$, and therefore $Q$ acts trivially on $P$ by coprime action (see \cite[Corollary 3.29]{isbook}). Thus $[Q,P]=1$ and we are done  in this case.

 %\textcolor{blue}{By It\^o-Michler and coprime action, we have that $\norm GP$ has a normal Sylow $q$-subgroup.}Therefore $[Q,P]=1$.

Suppose that  $L=\oh{p'}G>1$ and let  $P \in \syl pG$ and $Q\in \syl qG$ with $[Q,P] \sbs L$. By Hall--Higman 1.2.3 Lemma (see \cite[Theorem 3.21]{isbook}) we have that $\cent {G/L}{\oh {p}{G/L}}\subseteq \oh {p}{G/L}$. Since $QL/L\subseteq \cent {G/L}{PL/L}$ we conclude that $QL/L\subseteq \oh {p}{G/L}$ and hence $Q\subseteq L$. Thus $G/L$ is $q'$ and $G=L\norm GQ$, in particular $|G:\norm GQ|$ is not divisible by $p$.
We claim that all the irreducible characters in ${\rm Irr}_{q'}(B_q(\norm GQ))$ have $p'$-degree. Indeed, let $\chi^*\in{\rm Irr}_{q'}(B_q(\norm GQ))$ and let $\theta^*\in{\rm Irr}_{q'}(B_q(\norm LQ))$ under $\chi^*$. Let $\theta\in{\rm Irr}_{q'}(B_q(L))$ be the pre-image of $\theta^*$ given by the bijection in Conjecture \ref{overMcKay}. Again using Conjecture \ref{overMcKay}, let $\chi\in{\rm Irr}(B_q(G)|\theta)$ be such that $f_\theta(\chi)=\chi^*$, so we know that 
$$\frac{\chi(1)}{\theta(1)}=\frac{\chi^*(1)}{\theta^*(1)}.$$ Then $\chi(1)/\theta(1)$ is  not divisible by $q$ and thus $\chi\in{\rm Irr}_{q'}(B_q(G))$. By hypothesis, $\chi(1)$ is not divisible by $p$. Hence $\chi^*(1)/\theta^*(1)$ is not divisible by $p$, and therefore, since $\theta^*$ is of $p'$-degree we have that $\chi^*(1)$ is not divisible by $p$, and the claim follows.

Let $X=\oh{q'}{\norm GQ}$. Then all the elements in ${\rm Irr}(\norm GQ/Q'X)={\rm Irr}_{q'}(B_q(\norm GQ))$ have degree not divisible by $p$.
By the It\^o-Michler theorem, we have that this group has a normal Sylow $p$-subgroup (which is a Sylow $p$-subgroup of $G$, since $|G:\norm GQ|$ is $p'$),
and therefore $P$ centralizes $Q/Q'$. By coprime action, $[P,Q]=1$.
\end{proof}

\section{Conjecture C and Simple Groups}

Note that the ``if" direction for simple groups follows from the work of Malle--Navarro in \cite{MN20} and their proof of Brauer's height-zero conjecture for principal blocks \cite{MN21}:
\begin{prop}\label{prop:Cif} Let $S$ be a simple group. If there exists $P\in\Syl_p(S)$ and $Q\in\Syl_q(S)$ such that $[P,Q]=1$, then $q$ does not divide $\chi(1)$ for all $\chi \in \irrp{p}{B_p(S)}$
and $p$ does not divide $\chi(1)$ for all $\chi \in \irrp{q}{B_q(S)}$.
\end{prop}
\begin{proof}
In \cite[Section 3]{MN20}, the condition that $[P,Q]=1$ for some $P\in\Syl_p(S)$ and $Q\in\Syl_q(S)$ is studied.  In particular, we see from the results there that this condition implies that $P$ and $Q$ are abelian, and hence $\irrp{p}{B_p(S)}=\irr{B_p(S)}$ and $\irrp{q}{B_q(S)}=\irr{B_q(S)}$ by the principal block case of Brauer's height-zero conjecture \cite{MN21}. Hence the statement follows from \cite[Theorem 3.6]{MN20}. 
\end{proof}

Next we consider the cases easily dealt with in GAP:
\begin{prop}\label{prop:Csporadic}
Conjecture C holds for sporadic simple groups, alternating and symmetric groups $\alt_n$ and $\sym_n$ with $n\leq 8$, the Tits group $\tw{2}\type{F}_4(2)'$, and groups of Lie type with exceptional Schur multipliers.
\end{prop}
\begin{proof}
This can be seen using GAP and its Character Table Library.  
\end{proof}

\subsection{Conjecture C for Alternating and Symmetric Groups}

Here we prove Conjecture C in the case of alternating groups $\alt_n$ and symmetric groups $\sym_n$. Note that it follows from \cite[Proposition 3.3]{MN20} that $[P,Q]\neq 1$ for every Sylow $p$-subgroup $P$ and Sylow $q$-Subgroup $Q$ of $\sym_n$ or $\alt_n$.
\begin{prop}\label{prop:Caltsym}
Let $G$ be an alternating or symmetric group $\alt_n$ or $\sym_n$ with $n\geq 9$ and let $p,q$ be primes dividing $|G|$. %such that $[P,Q]\neq 1$ for every Sylow $p$-subgroup $P$ and Sylow $q$-Subgroup $Q$ of $S$. 
Then either there exists $\chi\in\irrp{p}{B_p(S)}$ with degree divisible by $q$ or there exists $\chi\in\irrp{q}{B_q(S)}$ with degree divisible by $p$.
\end{prop}

\textit{The strategy:} We first recall some facts and give the basic idea of the proof.    The set $\irr{\sym_n}$ is  indexed by partitions of $n$, and two characters $\chi_\lambda, \chi_\mu$ corresponding to partitions $\lambda, \mu$ lie in the same $p$-block if $\lambda, \mu$ have the same $p$-core (and similar for $q$). In particular, writing $n=pm+b$ with $0\leq b<p$, the set $\irr{B_p(\sym_n)}$ consists of the characters $\chi_\lambda$ such that $\lambda$ has $p$-core $(b)$. Further, recall that the degree of the character $\chi_\lambda$ is given by the hooklength formula $\chi_\lambda(1)=\frac{n!}{\prod h_\lambda}$, where the denominator is the product of all hooklengths in the tableaux corresponding to the partition $\lambda$. Further, if $\lambda$ is not self-conjugate, then the corresponding character restricts irreducibly to $\alt_n$. 

So, our strategy will be (up to switching $p$ and $q$) to illustrate a non-self-conjugate partition $\lambda$ of $n$ with $p$-core $(b)$ such that the numerator in the hooklength formula has the same $p$-part as the denominator and larger $q$-part than the denominator.  Our proof will require several technical cases and analysis of the degrees given by the hooklength formula.

\textit{Setting notation:} Throughout our proof, we will assume without loss that $q<p$. We will write $mp=wq+r$ with $0\leq r<q$, and let $n=mp+b$ with $0\leq b<p$.  Note that we may assume that $m>1$, since otherwise a Sylow $p$-subgroup is abelian, and the result follows from \cite[Theorem 3.5]{GMV19}, together with the principal block version of Brauer's height zero conjecture \cite{MN21}. In studying the degrees of the characters that we construct, several expressions will appear repeatedly.  Hence we define once and for all:
\[Y:=\frac{(mp+b)\cdots(mp+1)}{b!}, \quad\quad Y':=\frac{(mp+b-1)\cdots(mp+1)}{(b-1)!},\]\[ Z:=\frac{(mp+r)\cdots(mp+1)}{r!}, \quad\hbox{ and }\quad Z':=\frac{(mp-1)\cdots(mp-r)}{r!}\]
which we see are each relatively prime to $p$.  

It will also be useful to set the $p$-adic and $q$-adic expansions of $mp$: Let 
\[mp=a_1q^{t_1}+\cdots a_kq^{t_k}=b_1p^{s_1}+\cdots b_{k'}p^{s_{k'}}\] with $t_i<t_{i+1}$, $s_j<s_{j+1}$, $0<a_i<q$, and $0<b_j<p$ for each appropriate value of $i,j$.  With these established, we further define
\[X:=\frac{\prod_{i=1}^{a_1q^{t_1}}(mp-i)}{\prod_{i=1}^{a_1q^{t_1}}i} \quad\hbox{ and }\quad X':=\frac{\prod_{i=1}^{b_1p^{s_1}}(mp-i)}{\prod_{i=1}^{b_1p^{s_1}}i}.
\] Note that $mp\neq a_1q^{t_1}$, since $p\nmid a_1$.  In the situation that the expression $X'$ becomes relevant, we will see that also $mp\neq b_1p^{s_1}$.  Finally, for an integer $x$, we will write $(x)_p$ (or just $x_p$ if it is clear) for the $p$-part of $x$.

\begin{proof}[Proof of Proposition \ref{prop:Caltsym}]
 
Keep the notation above.

(I) First, suppose that $r>0$.   If $b=0$, so that $n=mp$, then consider $\lambda=(1^{mp-r-1}, 1+r)$. Then $\chi_\lambda\in\irr{B_p(\sym_n)}$ and $\chi_\lambda(1)=Z'$, %has degree $\frac{n!}{mp\cdot r!\cdot (mp-r-1)!}=\frac{(mp-1)\cdots (mp-r)}{r!}$, 
which is $p'$ but divisible by $q$ since $mp-r=wq$ and $r<q$. 
If $b\neq 0$ and $0<r<b$, consider $\lambda=(1^{mp-r-1}, 1+r, b)$. If $0<b<r$, consider $\lambda=(1^{mp-r-1}, 1+b, r)$. In these cases, $\chi_\lambda\in\irr{B_p(\sym_n)}$ and $\chi_\lambda(1)= \displaystyle{Y\cdot Z'\cdot\frac{|b-r|}{mp-r+b}}$. 
  Note that the $q$-part of the numerator of $Y$ must be at least as large as the $q$-part of the denominator. (Each remainder modulo $b$ appears once as a factor in the numerator.) Hence, this character still has degree that is $p'$ and divisible by $q$, with the possible exception of if $q\mid b$ and $b!(wq+b)$ has larger $q$-part than $(mp+1)\cdots(mp+b)$. In the latter case, $(1^{mp}, b)$, giving degree $Y'$, works instead.

If $b=r>0$, note that $r+1<wq$, as otherwise we would have $r=q-1$ and $w=1=m$, contradicting our assumption that $m>1$.  Then let $\lambda=(1^r, r+1, wq-1)$, so that  
$\chi_\lambda(1)=Z\cdot Z'\cdot ({wq-r-1})/({2r+1})$, which is $p'$ and divisible by $q$ unless $2r+1=p$ and $p\mid (m-1)$ or if $2r+1=q$ and $q\nmid w$. In the latter cases, the partition $(1^{wq-2}, 1+r, 1+r)$ works, unless we were in the case $2r+1=p$ with $p\mid (m-1)$ and $r+1=q$ with $q\nmid w$. In this case, if $q\neq 2$ (and hence $r\neq 1$), take $\lambda=(1^{mp}, r)$, which lies in $B_p(\sym_n)$ and has degree $Y'=(mp+1)\cdots(mp+r-1)/(q-2)!$.  This is relatively prime to $p$, and is divisible by $q$ since there must be a number between $mp$ and $mp+q-1=mp+r$ divisible by $q$, but neither $mp=(w+1)q-1$ nor $n=mp+r=wq+2q-2$ can be divisible by $q$.  If $q=2$, we have $r=1, q=2, p=3$, and the character corresponding to $(1^{n-2}, 2)$ lies in $B_2(\sym_n)$ and has degree $n-1=3m=2w+1$, which is odd and divisible by $3$.

From now on, we may therefore assume that $r=0$, so that $mp=wq$.

(II) First, assume that $a_1q^{t_1}<b_1p^{s_1}$.  If $b=0$, so $n=mp=wq$, consider $\lambda=(1^{mp-a_1q^{t_1}-1}, 1+a_1q^{t_1})$. Then the corresponding degree is 
$X$, which we see is equal to \[X=\frac{\prod_{i=0}^{a_1q^{t_1}-1}(i+a_2q^{t_2}+\cdots a_kq^{t_k})}{\prod_{i=1}^{a_1q^{t_1}}i}
=
\frac{\prod_{i=1}^{a_1q^{t_1}}(b_1p^{s_1}-i+b_2p^{s_2}+\cdots b_{k'}p^{s_{k'}})}{\prod_{i=1}^{a_1q^{t_1}}i}.\] Note that the $q$-part of this is $q^{t_2}/q^{t_1}$, which is divisible by $q$.  Further, the $p$-part is $1$, since $a_1q^{t_1}<b_1p^{s_1}$ implies that the $p$-part of $-i+b_1p^{s_1}+\cdots+b_{k'}p^{s_{k'}}$ is the same as that of $i$ for $1\leq i\leq a_1q^{t_1}$.

 Now suppose that $b>0$, so $n=mp+b=wq+b$. If $a_1q^{t_1}\neq b$, consider either $(1^{mp-a_1q^{t_1}-1}, 1+a_1q^{t_1}, b)$ or $(1^{mp-a_1q^{t_1}-1}, 1+b, a_1q^{t_1})$, depending on whether $b$ is larger or smaller than $a_1q^{t_1}$.  
 Then the corresponding character lies in $B_p(\sym_n)$ and has degree $X\cdot Y \cdot  \displaystyle{\frac{|b-a_1q^{t_1}|}{n-a_1q^{t_1}}}$. %, where $X$ is the degree from the previous situation and $Y$ is as defined in (I).  
 Note that since $a_1q^{t_1}<b_1p^{s_1}$, the $p$-part of $|b-a_1q^{t_1}|_p\leq p^{s_1}$, and from this we see $|b-a_1q^{t_1}|$ and $n-a_1q^{t_1}=mp+b-a_1q^{t_1}$ have the same $p$-part.  Hence this characters is a member of $\irrp{p}{B_p(\sym_n)}$.  Further, it is still divisible by $q$, except possibly if $|b-a_1q^{t_1}|_q<(n-a_1q^{t_1})_q$. This can only happen if $(b)_q\geq q^{t_2}$.  
 In the latter case, consider again the partition $(1^{mp}, b)$.  The degree is $Y'$ and hence we have removed $\frac{mp+b}{b}$ from the expression  $Y$, which is $p'$ and, from before, has $q$-part of the numerator at least as large as that of the denominator. Since in our situation $(mp+b)_q=q^{t_1}<q^{t_2}\leq (b)_q$, this degree $Y'$ is also divisible by $q$.

Now assume $b=a_1q^{t_1}$. If $b+1\neq p$ or $p\nmid (m-1)$, we take $\lambda=(1^{mp-b-2}, b+1, b+1)$, with $\displaystyle{\chi_\lambda(1)= Y\cdot\frac{(mp-2)\cdots(mp-b-1)}{(b+1)!}}.$ Then $\chi_\lambda\in\irr{B_p(\sym_n)}$ and $p\nmid\chi_\lambda(1)$ due to the assumption $b+1\neq p$ or $m-1$ is not divisible by $p$.  Further, $\chi_\lambda(1)$ is divisible by $q$ since the $q$-part of the numerators of each of the two fractions is at least that of the denominators, as before, and in this case, the factor $(mp-b)$ is divisble by $q^{t_2}$, but no factor in the denominator is. 
Now, if $p=b+1=1+a_1q^{t_1}$ and $p\mid (m-1)$, this forces also $b_1p^{s_1}=p$, as $mp-p=\sum {b_ip^{s_i}} - p$ must be divisible by $p^2$. Here consider the partition $(1^{mp-p}, b+p)$, which gives a character with degree ${(mp+b-1)\cdots(mp-p+1)}/{(p+b-1)!}$ in $B_p(\sym_n)$. Note that the only factor in the numerator divisible by $p$ is $mp$, which is divisible by $p$ eactly once. Then since the denominator is divisible by $p$, we see this character lies in $\irrp{p}{B_p(\sym_n)}$. 
Further, since $mp-p+1=\sum_{i\geq 2}a_iq^{t_i}$ in this case, we  see $(mp-p+1+j)_q\geq (j)_q$ for $1\leq j\leq p+b-2$, and that $\displaystyle{\frac{mp-p+1}{p+b-1}=\frac{\sum_{i\geq 2}a_iq^{t_i}}{2a_1q^{t_1}}}$ is divisible by $q$ except possibly if $q=2$ and $t_2=t_1+1$. In the latter case, the character corresponding to $(1^{mp-1}, b+1)$, which has degree $\displaystyle{\frac{mp\cdot\prod_{i=1}^{2^{t_1}-1} (mp+i)}{2^{t_1}!}}$ lies in $\irrp{2}{B_2(\sym_n)}$ and has degree divisible by $p$.

(III) Finally, suppose that $a_1q^{t_1}>b_1p^{s_1}$.  Note here that $mp\neq b_1p^{s_1}$, as $mp\geq a_1q^{t_1}$.   If $b=0$, then reversing the roles of $p$ and $q$ in the corresponding case in (II) above yields a character in $B_q(\sym_n)$ with degree $X'$, which is relatively prime to $q$ but divisible by $p$. %We will write $X'$ for this degree . 

Hence we assume $b>0$, so $n=mp+b=wq+b$. Note here that $b=w'q+b'$ for some integers $w', b'$ with $0\leq b'<q$, and $B_q(\sym_n)$ consists of those characters whose corresponding partitions have $q$-core $(b')$.

 Now, the partition $(1^{mp-b_1p^{s_1}-1}, b+1, b_1p^{s_1})$ gives a character in $B_q(\sym_n)$ with degree $\displaystyle{Y\cdot X'\cdot \frac{b_1p^{s_1}-b}{n-b_1p^{s_1}}}$.
  Note that the third factor is not divisible by $p$ nor $q$, since $p\nmid b$ and $b_1p^{s_1}-b<a_1q^{t_1}$ so $b_1p^{s_1}-b$ and $n-b_1p^{s_1}=mp-(b_1p^{s_1}-b)$ have the same $q$-part.  Further, $p\nmid Y$, and also $q\nmid Y$ as long as $b<(q-a_1)q^{t_1}$. So if $b<(q-a_1)q^{t_1}$, this character lies in $\irrp{q}{B_q(\sym_n)}$ with degree divisible by $p$.

 So, we now assume that $q\mid Y$, so $b\geq (q-a_1)q^{t_1}$.  Then the partition $(b+1, mp-1)$ corresponds to a character in $B_p(\sym_n)$ with degree $\displaystyle{Y\cdot \frac{mp-(b+1)}{b+1}}$, which is divisible by $q$ and is relatively prime to $p$  if $b+1\neq p$ or $p\nmid (m-1)$.  Hence we may now assume that further $p=b+1$ and $p\mid (m-1)$.  
 Then setting $\lambda=(1^{mp}, b)$ yields $\chi_{\lambda}\in\irr{B_p(\sym_n)}$ and $\chi_{\lambda}(1)=Y'$, which is prime to $p$.  Since $q\mid Y$, we have $q\mid Y'$ unless $q\mid b$ and $(mp+b)_q>(b)_q$, which forces $b=(q-a_1)q^{t_1}$.
 
 So we are reduced to the case $q\nmid Y'$, $b=(q-a_1)q^{t_1}=p-1$, and $p\mid (m-1)$. Then the partition $(1^{mp-1}, 1+b)$ gives a character in $B_q(\sym_n)$ with degree $Y'\cdot {mp}/{b}$, which is divisible by $p$ but is not divisible by $q$ since the $q$-part of both $b$ and $mp$ are $q^{t_1}$.

Finally, note that in all cases, the $\lambda$ described is not self-conjugate, and hence the characters restrict irreducibly to $\alt_n$, completing the proof.
\end{proof}

\subsection{Conjecture C for Simple Groups of Lie Type}

Let $r_0$ be a prime and $r:=r_0^a$ be some power of $r_0$.  For $p$ another prime, we denote by $d_p(r)$ the order of $r$ modulo $p$, respectively modulo $4$, if $p$ is odd, respectively $p=2$. Here we will prove the remaining direction of Conjecture C for simple groups of Lie type 
  $S$ defined over $\FF_{r}$.

Several cases here also follow quickly from \cite{MN20}:

\begin{prop}\label{prop:Cdefabetc}
Let $S$ be a simple group of Lie type defined over $\FF_r$, and let $p,q$ be primes such that $[P,Q]\neq 1$ for every Sylow $p$-subgroup $P$ and Sylow $q$-subgroup $Q$ of $S$ but that at least one of the following conditions holds:
\begin{enumerate}
\item $r_0\in\{p,q\}$;
\item a Sylow $p$-subgroup or a Sylow $q$-subgroup of $S$ is abelian; or
\item $d_p(r)\neq d_q(r)$;
\end{enumerate}
 Then either there exists $\chi\in\irrp{p}{B_p(S)}$ with degree divisible by $q$ or there exists $\chi\in\irrp{q}{B_q(S)}$ with degree divisible by $p$.
\end{prop}
\begin{proof}
First suppose that $r_0\in\{p,q\}$  and without loss, say $r_0=p$. Then $\irr{B_p(S)}=\irr{S}\setminus\{\mathrm{St}_S\}$. Let $G$ be a quasisimple group of Lie type of simply connected type such that $G/\zent{G}=S$. As in \cite[Theorem 5.1]{MN20}, we let $\chi\in\irr{S}$ be of positive $q$-height, so that it must lie in $B_q(S)$. But then it must be that $\chi$ is the deflation of some $\bar\chi\in\irr{G}$ in a Lusztig series $\mathcal{E}(G, s)$, where $s$ is a semisimple element of the dual group $G^\ast$ that does not centralize a Sylow $q$-subgroup of $G^\ast$ and which lies in $O^{p'}(G^\ast)$. Considering a semisimple character $\chi_s$ in this series instead, we see that $\chi_s(1)=[G^\ast:\cent{G^\ast}{s}]_{p'}$, so $\chi_s$ still has degree divisible by $q$ but lies in $\irrp{p}{B_p(S)}$ as a character of $S$. 

Next, suppose that $r_0\not\in\{p,q\}$ and that a Sylow $p$-subgroup of $S$ is abelian. Then $\irrp{p}{B_p(S)}=\irr{B_p(S)}$ by Brauer's height zero conjecture for principal blocks \cite{MN21}, and hence the proof of \cite[Theorem 5.1]{MN20} yields the desired character.

Finally, assume that $r_0\not\in\{p,q\}$, that no Sylow $p$- or $q$- subgroup of $S$ is abelian, and that $d_p(r)\neq d_q(r)$.
Note that the assumption $[P,Q]\neq 1$ for every choice of Sylow $p$- and Sylow $q$-subgroups of $S$ (and hence analogously for $G$) implies that also $[P^\ast, Q^\ast]\neq 1$ for any Sylow $p$- and $q$- subgroups of the dual group $G^\ast$. (Again, this is pointed out already in \cite[Theorem 5.1]{MN20}.)  
Now, let $d_p(r)\leq d_q(r)$ and %we claim that further $[P^\ast, \zent{Q^\ast}]\neq 1$ for any $P^\ast\in\Syl_p(G^\ast)$ and $Q^\ast\in\Syl_q(G^\ast)$.  Otherwise, 
suppose that there exists $1\neq s\in\zent{Q^\ast}$ for some $Q^\ast\in\Syl_q(G^\ast)$ such that $s$ centralizes a Sylow $p$-subgroup of $G^\ast$. Then the same argument as in the first two paragraphs of \cite[Proposition 3.5]{MN20}, but now applied to $G^\ast$, yields that $d_p(r)=d_q(r)$, a contradiction. Hence, given $1\neq s\in\zent{Q^\ast}$, we have $\cent{G^\ast}{s}$ does not contain a Sylow $p$-subgroup $P^\ast$ of $G^\ast$. Then the corresponding semisimple character $\chi_s$ of $G$ has degree divisible by $p$ but lies in $\irrp{q}{B_q(G)}$.  Further, arguing as in \cite[Theorem 3.5]{GSV19} shows that such $s$ can be chosen so that $\chi_s$ is trivial on $\zent{G}$, completing the proof.
\end{proof}

Our task is now to prove Conjecture C in the case that neither $p$ nor $q$ is the defining characteristic, $d_p(r)=d_q(r)$, and no Sylow $p$- or $q$-subgroup of $S$ is abelian.  We begin with the case of exceptional groups of Lie type, by which we mean the groups $S=\type{G}_2(r)$, $\type{F}_4(r)$, $\type{E}_6^\epsilon(r)$, $\type{E}_7(r)$, $\type{E}_8(r)$, $\tw{3}\type{D}_4(r)$, $\tw{2}\type{G}_2(r)$, $\tw{2}\type{F}_4(r)$, and $\tw{2}\type{B}_2(r)$.

\begin{prop}\label{prop:Cexceptional}
Let $S$ be an exceptional simple group of Lie type  defined over $\FF_r$, and let $p,q$ be primes such that $[P,Q]\neq 1$ for every Sylow $p$-subgroup $P$ and Sylow $q$-subgroup $Q$ of $S$. Then either there exists $\chi\in\irrp{p}{B_p(S)}$ with degree divisible by $q$ or there exists $\chi\in\irrp{q}{B_q(S)}$ with degree divisible by $p$.
\end{prop}
\begin{proof}
By Proposition \ref{prop:Cdefabetc}, we may assume that $r$ is not a power of $p$ nor $q$, that no Sylow $p$- or $q$- subgroup of $S$ is abelian, and that $d_p(r)=d_q(r)$. Let $d:= d_p(r)=d_q(r)$.  With these constraints, we see that $S$ is not of Suzuki or Ree type, that $p$ and $q$ are at most 7, and that $d$ is a regular number.  Hence we see that the principal blocks $B_p(S)$ and $B_q(S)$ are the the unique blocks of $S$ containing $p'$-, respectively, $q'$-degree unipotent characters (see, e.g., \cite[Lemma 3.6]{RSV21}). Under these conditions, we see by observing the explicit list of unipotent character degrees in \cite[Section 13.9]{carter} that there exists a unipotent character $\chi$ satisfying either $\chi\in\irrp{p}{B_p(S)}$ and $q\mid \chi(1)$ or $\chi\in\irrp{q}{B_q(S)}$ and $p\mid\chi(1)$.
\end{proof}

We next consider the case of linear and unitary groups. 
\begin{prop}\label{prop:CtypeA}
Let $S=\PSL_n^\epsilon(r)$ with $n\geq 2$, and   let $p\neq q$ be primes such that $[P,Q]\neq 1$ for every Sylow $p$-subgroup $P$ and Sylow $q$-subgroup $Q$ of $S$. Then either there exists $\chi\in\irrp{p}{B_p(S)}$ with degree divisible by $q$ or there exists $\chi\in\irrp{q}{B_q(S)}$ with degree divisible by $p$.
\end{prop}

\begin{proof}
By Proposition \ref{prop:Csporadic}, we may assume $S$ is not isomorphic to a sporadic or alternating group and that $S$ has a nonexceptional Schur multiplier.  Further, by Proposition \ref{prop:Cdefabetc}, we may assume that $r$ is not a power of $p$ nor $q$, that no Sylow $p$- or $q$- subgroup of $S$ is abelian, and that $d_p(r)=d=d_q(r)$. 

%First consider the case $S=\PSL_n^\epsilon(r)$.  

Write $G:=\SL_n^\epsilon(r)$, $\wt{G}:=\GL_n^\epsilon(r)$, and $e:=d_p(\epsilon r)=d_q(\epsilon r)$. Note that $\wt{G}^\ast\cong \wt{G}$, so we will make this identification. Let $\wt{P}\in\Syl_p(\wt{G})$ and $\wt{Q}\in\Syl_q(\wt{G})$.  It suffices to show that (up to switching $p$ and $q$) there is some $s\in\zent{\wt{P}}\cap G$ such that $\cent{\wt{G}}{s}$ is not divisible by $|\wt{Q}|$ and $sz$ is not $\wt{G}$-conjugate to $s$ for any $1\neq z\in\zent{\wt{G}}$. 
(Indeed, that $s\in G=\OO^{r_0'}(\wt{G})$ implies that the corresponding semisimple character $\chi_s$ is trivial on $\zent{\wt{G}}$ and that $sz$ is not conjugate to $s$ for nontrivial $z\in\zent{\wt{G}}$ implies that $\chi_s$ is irreducible on restriction to $G$ using e.g. \cite[Proposition 2.7]{SFT22} and \cite[Lemma 1.4]{RSV21};  the remaining conditions imply $\chi_s\in\irr{B_p(\wt{G})}$ using \cite[Corollary 3.4]{Hiss90} and $\chi_s(1)$ is $p'$ but divisible by $q$ since $\chi_s(1)=[\wt{G}:\cent{\wt{G}}{s}]_{r_0'}$.)

Now, if $2\not\in\{p,q\}$, we see using 
the results of Weir \cite{weir} that  $\wt{P}$ and $\wt{Q}$ are naturally isomorphic to the corresponding Sylow subgroups of $\GL_{we}^\epsilon(r)$, embedded naturally into $\GL_{we}^\epsilon(r)\times \GL_{b}^\epsilon(r)\leq \GL_n^\epsilon(r)$ where $n=we+b$ with $0\leq b<e$. The results of Carter--Fong \cite{carterfong} yield the same when $p$ or $q$ is $2$, except that if $(e,b)=(2,1)$, then $|\GL_{b}^\epsilon(r)|$ is divisible by $2$ exactly once, and the Sylow $2$-subgroup of $\GL_n(r)$ in this case is that of  $\GL_{we}^\epsilon(r)\times \GL_{b}(r)$.  %, and further there is a natural bijection between the principal $p$- and $q$-blocks of $\GL_{we}^\epsilon(r)$ and those of $\wt{G}$.  Hence we may further assume without loss that $e\mid n$, and write $n=ew$.

Let $w=a_1p^{t_1}+a_2p^{t_2}+\cdots+a_kp^{t_k}=b_1q^{m_1}+\cdots+b_{k'}q^{m_{k'}}$ be the $p$-adic and $q$-adic expansions of $w$, with $t_1\leq\cdots\leq t_k$; $m_1\leq\cdots \leq m_{k'}$; $1\leq a_i<p$ for each $1\leq i\leq k$; and $1\leq b_j<q$ for each $1\leq j\leq k'$.  By \cite{weir, carterfong}, we have $\wt{P}\cong P_1^{a_1}\times \cdots\times P_k^{a_k} \times X$, where $P_i$ is a Sylow $p$-subgroup of $\GL_{ep^{t_i}}^\epsilon(r)$ (which, if $p$ is odd, is a Sylow $p$-subgroup of $\GL^\epsilon_{p^{t_i}}(r^e)$ embedded naturally) for each $1\leq i\leq k$ and $X\in\Syl_p(\GL_b^\epsilon(r))$ is isomorphic to $C_2$ if $(p,e,b)= (2,2,1)$ and is trivial otherwise.  Here we view $\prod \GL_{ep^{t_i}}^\epsilon(r)$ as the natural diagonally-embedded subgroup.  Similarly, $\wt{Q}\cong Q_1^{b_1}\times\cdots\times Q_{k'}^{b_k}\times Y$ where $Q_j\in \Syl_q(\GL_{eq^{m_j}}^\epsilon(r))$ for $1\leq j\leq k'$ and $Y\in\Syl_q(\GL_b^\epsilon(r))$.

Without loss, assume $a_1p^{t_1}<b_1q^{m_1}$. (Note that we cannot have $a_1p^{t_1}=b_1q^{m_1}$, as this would contradict that either $a_1<p<q$ or $b_1<q<p$.) 
Let $x\in \zent{P_1}$ have no eigenvalues equal to $1$ (indeed, this can be done by taking $x$ as an element of a Sylow $p$-subgroup of $\zent{\GL_{p^{t_1}}^\epsilon(r^e)}$
 embedded naturally into $\GL_{ep^{t_1}}^\epsilon(r)$) and consider the element
  $s=\mathrm{diag}(x, \ldots, x, I_{n-ea_1p^{t_1}})\in\zent{\wt{P}}$, with $a_1$ copies of $x$. In fact, taking $x$ (and hence its eigenvalues) to have order $p$, we obtain $\det(x)=1$ and hence $s\in G$. 
  Here $\cent{\wt{G}}{s}=\cent{\GL_{a_1ep^{t_1}}^\epsilon(r)}{\mathrm{diag}(x,\ldots, x)}\times \GL_{n-ea_1p^{t_1}}^\epsilon(r)$.  
  Now, by considering the structure, and hence size, of $\wt{Q}$ (namely, each $Q_j$ is a wreath product $Q'\wr C_q$ where $Q'$ is a Sylow $q$-subgroup of $\GL_{eq^{m_j-1}}^\epsilon(r)$), we see $\GL_{a_1ep^{t_1}}^\epsilon(r)\times \GL_{n-ea_1p^{t_1}}^\epsilon(r)$ is not divisible by $|\wt{Q}|$.
  Further, by considering the block sizes, we see that $sz$ and $z$ cannot have the same eigenvalues (and hence they cannot be $\wt{G}$-conjugate) for any nontrivial scalar matrix $z\in\zent{\GL_n^\epsilon(r)}$.  This completes the proof. % when $S=\PSL_n^\epsilon(r)$. % \textcolor{red}{ Now suppose $p=2$//////}\textcolor{red}{///for $p=2$, if $e=1$ or remainder $b=0$, same should work using carter-fong instead of Weir; also unique unipotent $2$-block; if $b=1, e=2$, just need consider $n=2m+1$ instead of $2m$ but should still work out////}
 \end{proof}
 
 We next consider the remaining classical types, for which the proof is very similar to the linear and unitary case.

\begin{prop}\label{prop:Cclassical}
Let $S=\PSp_{2n}(r)$ with $n\geq 2$, $\POmega_{2n+1}(r)$ with $n\geq 3$, or $\POmega_{2n}^\epsilon(r)$ with $n\geq 4$.  Let $p\neq q$ be primes such that $[P,Q]\neq 1$ for every Sylow $p$-subgroup $P$ and Sylow $q$-Subgroup $Q$ of $S$. Then either there exists $\chi\in\irrp{p}{B_p(S)}$ with degree divisible by $q$ or there exists $\chi\in\irrp{q}{B_q(S)}$ with degree divisible by $p$.
\end{prop}

 \begin{proof}
 As before, we may assume $S$ is not isomorphic to a sporadic or alternating group, that $S$ has a nonexceptional Schur multiplier, $r$ is not a power of $p$ nor $q$,  no Sylow $p$- or $q$- subgroup of $S$ is abelian, and that $d_p(r)=d=d_q(r)$.

First we set some notation.  We define $H_n:=\Sp_{2n}(r)$, $\SO_{2n+1}(r)$, and $\SO_{2n}^\epsilon(r)$ in the cases $S=\PSp_{2n}(r)$, $\POmega_{2n+1}(r)$, and $\POmega_{2n}^\epsilon(r)$, respectively. Let $H:=H_n$ and let $\Omega:=\OO^{r_0'}(H)$ so that $\Omega$ is perfect and $S=\Omega/\zent{\Omega}$. Note that the dual groups are $H_n^\ast =  \SO_{2n+1}(r)$, $\Sp_{2n}(r)$, and $\SO_{2n}^\epsilon(r)$, respectively, and we will write $H^\ast:=H_n^\ast$. Note that $\zent{\Omega}\leq \zent{H}$ and that $H/\Omega$ and $\zent{H}$ are $2$-groups.  Let $\wt{P}$ and $\wt{Q}$ be Sylow $p$- and $q$-subgroups of $H^\ast$.

In this situation, it suffices to show that (up to switching $p$ and $q$) there is some $s\in\zent{\wt{P}}$ such that $\cent{\wt{G}}{s}$ is not divisible by $|\wt{Q}|$, using similar reasoning to the above case.  Indeed, if $p$ is odd, then $\wt{P}$ may be considered as a Sylow $p$-subgroup of $\OO^{r_0'}(H^\ast)$ and $sz$ cannot be $H^\ast$-conjugate to $s$ for any $1\neq z\in\zent{H^\ast}$ since $\zent{H^\ast}$ is a 2-group, so a corresponding semisimple character $\chi_s$ of $H$ is trivial on $\zent{H}$ and restricts irreducibly to $\Omega$.  If instead $p=2$, then such a character $\chi_s$  would have odd degree, and therefore restrict irreducibly to $\Omega$ since $H/\Omega$ is a $2$-group.  Then since $\Omega$ is perfect,  $\zent{\Omega}$ is a $2$-group, and $\chi_s$ has odd-degree, this forces $\chi_s$ to be trivial on $\zent{\Omega}$.  Further, as before, $\chi_s\in\irr{B_p(H)}$ in either case.

Assume first that $p$ and $q$ are odd.  In these cases, the work of Weir \cite{weir} again describes the structure of $\wt{P}$ and $\wt{Q}$, building off of the case of linear groups.  If $H^\ast=\SO_{2n+1}(r)$ or $\Sp_{2n}(r)$, we have Sylow $p$- and $q$-subgroups are already Sylow subgroups of $\GL_{2n+1}(r)$ (and hence of $\GL_{2n}(r)$) when $d$ is even, and are Sylow subgroups of the naturally-embedded $\GL_n(r)$ if $d$ is odd.  For these cases, let $e:=d_p(r^2)=d_q(r^2)$,  write  $n=ew+b$ with $0\leq b<e$, and let $w=a_1p^{t_1}+a_2p^{t_2}+\cdots+a_kp^{t_k}=b_1q^{m_1}+\cdots+b_{k'}q^{m_{k'}}$ be the $p$-adic and $q$-adic expansions of $w$ as before. Again without loss, we assume $a_1p^{t_1}<b_1q^{m_1}$.
 In particular, $\wt{P}$ and $\wt{Q}$ are again isomorphic to Sylow subgroups of $H_{ew}^\ast$ and of the form $\wt{P}\cong P_1^{a_1}\times\cdots\times P_k^{a_k}$, where each $P_i$ is a Sylow $p$-subgroup of $\GL_{dp^{t_i}}(r)$ and can be identified with a Sylow $p$-subgroup of  $H_{ep^{t_i}}^\ast$, and similar for $\wt{Q}$. %(But here the embedding is described through the appropriate general linear group.)
 As before, let $x\in\zent{P_1}$ with no eigenvalues equal to 1 and let $s=(x,\cdots x, 1, \ldots 1)\in\zent{\wt{P}}$ with $a_1$ copies of $x$. Then we can see from the centralizer structure of semisimple elements that $\cent{H^\ast}{s}\cong\cent{H^\ast_{a_1ep^{t_1}}}{(x,\ldots, x)}\times H^\ast_{n-ea_1p^{t_1}}$. Since the Sylow $q$-subgroups of $H_{a_1ep^{t_1}}^\ast$ and $H_{n-ea_1p^{t_1}}^\ast$  can be identified with Sylow subgroups of linear groups in an analogous way as for $H^\ast$, depending on whether $d$ was even or odd, we have $|\wt{Q}|\nmid |\cent{H^\ast}{s}|$ for the same reason as the case of linear groups above.

% (with the exception of the case $\SO_{2n}^\epsilon(r)$, in which case $P_i\in \Syl_p(\SO_{2ep^{t_i}}^+(r))$ or $\Syl_p(\SO_{2ep^{t_i}}^-(r))$)  Here, however, the embedding is described through some general linear group.  Namely, in the case that  
If $H^\ast=\SO_{2n}^\epsilon(r)$, then we have embeddings $\SO_{2n-1}(r)\leq H^\ast\leq \SO_{2n+1}(r)$, and $\wt{P}$ and $\wt{Q}$ are both Sylow subgroups of either $\SO_{2n-1}(r)$ or $\SO_{2n+1}(r)$.  In this case, letting $m\in\{n, n-1\}$ so that $\wt{P}, \wt{Q}$ are  Sylow subgroups of $\SO_{2m+1}(r)$ and now writing $m=ew+b$ with $w$ written with $p$- and $q$-adic expansions as before, $\wt{P}$ can again be written $\wt{P}\cong P_1^{a_1}\times\cdots\times P_k^{a_k}$ with each $P_i$ a Sylow subgroup of $\GL_{dp^{t_i}}(r)$, which in this case can also be identified with a Sylow $p$-subgroup of either $\SO_{2ep^{t_i}}^+(r)$ or $\SO_{2ep^{t_i}}^-(r)$.  From here, arguing similar to before, we obtain an element $s\in\zent{\wt{P}}$ with $\cent{H^\ast}{s}$ isomorphic to a subgroup of $\cent{\GO_{2a_1ep^{t_1}}^\pm(r)}{\mathrm{diag}(x,\ldots, x)}\times \GO^\pm_{2(n-ea_1p^{t_1})}(r)$. Since $q$ is odd, we again see in the same way as above that $|\wt{Q}|$ does not divide $|\cent{H^\ast}{s}|$. 
  %and by \cite[Theorem 5.17]{malle19} \textcolor{red}{careful: $GO$ vs $\SO$ ...but should still work since restricted chars still $p'$ and div by $q$ since $[GO:\SO]$ power of $2$}, there is a bijection between characters of the corresponding principal blocks.  Hence we again assume that $n=ew$. 

%As before, let $x\in\zent{P_1}$ with no eigenvalues equal to 1 and let $s=(x,\cdots x, 1, \ldots 1)\in\zent{\wt{P}}$ with $a_1$ copies of $x$. Then we can see from the centralizer structure of semisimple elements that $\cent{H^\ast}{s}\cong\cent{H^\ast_{a_1ep^{t_1}}}{(x,\ldots, x)}\times H^\ast_{n-ea_1p^{t_1}}$, 
%with the exception of the case $H^\ast=\SO^\epsilon_{2n}(r)$, in which case $\cent{H^\ast}{s}$ is isomorphic to a subgroup of $\cent{\GO_{2a_1ep^{t_1}}^\epsilon(r)}{\mathrm{diag}(x,\ldots, x)}\times \GO^\epsilon_{2(n-ea_1p^{t_1})}(r)$. 
% Since the Sylow $q$-subgroups of $H_{a_1ep^{t_1}}^\ast$ and $H_{n-ea_1p^{t_1}}^\ast$ (or the corresponding groups of the form $\GO_{2m}^\epsilon(r)$, since $q$ is odd) can be identified with Sylow subgroups of linear groups in an analogous way as for $H^\ast$, depending on whether $d$ was even or odd, we have $|\wt{Q}|\nmid |\cent{H^\ast}{s}|$ for the same reason as the case of linear groups above. ///double check everything for $\SO_{2n}$ goes through in the case $\SO_{2n-1}$///

We are finally left with the case that $2\in\{p,q\}$. Let $\hat H^\ast$ denote the group $\GO_{2n+1}(r), \Sp_{2n}(r)$, or $\GO_{2n}^\epsilon(r)$ respectively, so that $[\hat H^\ast: H^\ast]$ divides $2$. Note that if $p=2$ and $H^\ast\neq \hat H^\ast$, then $\wt{P}$ is index-2 in a Sylow $2$-subgroup $\hat P$ of $\hat H^\ast$, which are again described by Carter--Fong \cite{carterfong}.  Here in the case of $\GO_{2n+1}(r)$ or $\Sp_{2n}(r)$, writing $n=2^{t_1}+\cdots + 2^{t_k}$ for the $2$-adic expansion with $t_1\leq\ldots\leq t_k$, we have $\hat P\cong P_1\times\cdots\times P_k$, where $P_i$ is a Sylow $2$-subgroup of $\GO_{2\cdot 2^{t_i}+1}(r)$, respectively, $\Sp_{2\cdot 2^{t_i}}(r)$. %, and $X$ is a Sylow $2$-subgroup of $\Sp_2(r)$ if $(e,b)=(2,1)$ and is trivial otherwise. 
%n=2w+b
%b=1 -> n=2w+1 -> w=\sum_{i\neq 1} 2^{t_i-1} (takes out the 2^0 piece)
In the case $\hat H^\ast=\GO_{2n}^\epsilon(r)$, we have $\hat{P}$ is either a Sylow $2$-subgroup of $\GO_{2n+1}(r)$, embedded as before, or of the form $P_0\times C_2\times C_2$, where $P_0$ is a Sylow $2$-subgroup of $\GO_{2n-1}(r)$.  
From here, we may argue analogously to before, keeping in mind that when $p=2$, choosing $x\in\zent{P_1}$ to have $2^{t_1+1}$ eigenvalues $-1$ yields an element of determinant 1, and hence an element of $\wt{P}=\hat P\cap H^\ast.$
\end{proof}

Conjecture C for simple groups (and Theorem E) now follows from Propositions \ref{prop:Cif}-\ref{prop:Cclassical}.

\end{document}